\documentclass[11pt]{article}

\usepackage[square,authoryear]{natbib}
\usepackage{marsden_article}

\makeatletter
\def\blfootnote{\xdef\@thefnmark{}\@footnotetext}
\makeatother

\newcommand{\bracleft}{\left[\!\!\left[}
\newcommand{\bracright}{\right]\!\!\right]}

\newcommand{\pairleft}{\left<\!\left<}
\newcommand{\pairright}{\right>\!\right>}

\newcommand{\lie}{\pounds}

\begin{document}
\title{On the Geometry of Multi-Dirac Structures and Gerstenhaber Algebras}
\author{Joris Vankerschaver$^{a, b}$, Hiroaki Yoshimura$^c$,
Melvin Leok$^a$}

\maketitle

\blfootnote{\noindent %
$^a$ Department of Mathematics, University of California at San Diego, 9500 Gilman Drive, San Diego CA 92093, USA.  Email: \{jvankers, mleok\}@math.ucsd.edu. \\
$^b$ Department of Mathematics, Ghent University, Krijgslaan 281, B-9000 Ghent, Belgium. \\
$^c$ Applied Mechanics and Aerospace Engineering, Waseda University, Okubo, Shinjuku, Tokyo, 169-8555, Japan.  Email:  yoshimura@waseda.jp.
}


\begin{center}
{\it We dedicate this paper to the memory of \\ Jerrold E. Marsden, our friend and mentor.}
\end{center}

\begin{abstract}
	In a companion paper, we introduced a notion of multi-Dirac structures, a graded version of Dirac structures, and we discussed their relevance for classical field theories.  In the current paper we focus on the geometry of multi-Dirac structures.  After recalling the basic definitions, we introduce a graded multiplication and a multi-Courant bracket on the space of sections of a multi-Dirac structure, so that the space of sections has the structure of a Gerstenhaber algebra. We then show that the graph of a $k$-form on a manifold gives rise to a multi-Dirac structure and also that this multi-Dirac structure is integrable if and only if the corresponding form is closed.  Finally, we show that the multi-Courant bracket endows a subset of the ring of differential forms with a graded Poisson bracket, and we relate this bracket to some of the multisymplectic brackets found in the literature.
\end{abstract}


\section{Introduction}

Dirac structures were introduced by \cite{Courant1990} as a simultaneous generalization of pre-symplectic and Poisson structures.  It was quickly realized that these structures have remarkable properties that make them into a fundamental tool in geometry as well as in classical mechanics, where they are used to describe mechanical systems with symmetry, constraints, or interconnected systems; see \cite{MarsdenDirac, YoMa2006, DaSc1999} for more details.

The present paper arose out of the need for a similar structure to describe \emph{classical field theories}.  Based on a variational principle for field theory, in \cite{VaYoLeMa2011} we defined a multi-Dirac structure of degree $n$ on a manifold to be a subbundle of the exterior algebra 
\begin{equation} \label{gradedbundle}
	E := \mbox{$\bigwedge^\bullet$} (TZ) \oplus \mbox{$\bigwedge^ \bullet $} (T^\ast Z)
\end{equation}
 of multivector fields and forms satisfying a certain {\it maximally isotropic property} reminiscent of the usual maximally isotropic property of Dirac structures. The advantages of the use of multi-Dirac structures for the description of classical field theories is that they yield the field equations in implicit form and are hence well-suited for field theories with constrained momenta due to gauge symmetry or nonholonomic constraints.

In the current paper, which should be seen as a companion paper to \cite{VaYoLeMa2011}, we describe the mathematical properties of multi-Dirac structures in somewhat greater detail.  We recall the definition of a multi-Dirac structure $D$ and show that there exists a graded multiplication and a graded bracket (referred to as the {\it multi-Courant bracket}) on the space of sections of $D$ and that the latter is endowed with the structure of a {\it Gerstenhaber algebra} with respect to these two operations.  The multi-Courant bracket then allows us to define a multi-Poisson bracket on a distinguished subset of the space of forms and we show that this bracket satisfies the graded Jacobi identity up to exact forms, while the multi-Courant bracket satisfies the graded Jacobi identity exactly, as required by the definition of a Gerstenhaber algebra.  

The multi-Dirac structures introduced in this paper also turn out to have unexpected links with other kinds of Dirac-like structures.   Not only do multi-Dirac structures include standard Dirac structures developed by \cite{Courant1990} as a special case, it was also shown by \cite{Zambon2010} that multi-Dirac structures are in a one-to-one correspondence with \emph{higher-order Dirac structures}.   While this correspondence is relatively easy to state, the implications are not entirely straightforward: while multi-Dirac structures give rise to Gerstenhaber algebras, higher-order Dirac structures are related to \emph{$L_\infty$-algebras}, an observation due to \cite{Rogers2010b} for multisymplectic forms and extended to higher-order Dirac structures by \cite{Zambon2010}.  The relation between both algebraic structures, however, is as yet unclear.

Furthermore, the multi-Dirac structures in this paper generalize the ``multi-Poisson'' bracket of \cite{CaIbLe1996} on the space of Hamiltonian forms.  However, while the latter satisfies the graded Jacobi identity only up to exact forms,  our structures (through the associated Gerstenhaber algebra) satisfy the graded Jacobi and Leibniz identities exactly, making them easier to understand.  Finally, while Gerstenhaber structures were studied previously in field theory (see \cite{Kanatchikov1995}), our approach does not rely on any extraneous structures for its definition, such as the choice of a background metric or connection, and is therefore more canonical.

\paragraph{Conventions.}

Throughout this paper, all manifolds and maps are smooth.  We denote the exterior algebra of the tangent and cotangent bundle of $Z$ by $\bigwedge^ \bullet (TZ)$ and $\bigwedge^ \bullet (T^\ast Z)$, respectively.  Multivector fields are sections of the former, while forms are sections of the latter.  For the signs in the definition of the Schouten-Nijenhuis bracket on $\bigwedge^ \bullet (TZ)$, we use the convention of \cite{FoPaRo2005} and \cite{Marle1997}.  A brief overview of the most important properties of multivector fields can be found in appendix~\ref{sec:lie}.

\paragraph{Acknowledgements.}  

We are very grateful to Henrique Bursztyn, Frans Cantrijn,  Mark Gotay, Alberto Ibort, Eduardo Mart\'{\i}nez, Chris Rogers, Alan Weinstein and Marco Zambon who kindly provided several very useful remarks.  
  
The research of H. Y. is partially supported by JSPS Grant-in-Aid 20560229, JST-CREST and Waseda University Grant for SR 2010A-606.  M. L. and J. V. are supported by NSF CAREER award DMS-1010687.  J. V. is on leave from a Postdoctoral Fellowship of the Research Foundation--Flanders (FWO-Vlaanderen) and would like to thank JSPS for financial support during a research visit to Waseda University, where part of this work was carried out. This work is part of the Irses project GEOMECH (nr. 246981) within the 7th European Community Framework Programme.

\section{The Graded Courant Bracket}

In this section, we show that the space of sections of the exterior algebra \eqref{gradedbundle} of multivector fields and forms on a manifold $Z$ is equipped with a graded multiplication and a bracket of degree -1.  Furthermore, we show that closed forms on $Z$ give rise to automorphisms of these structures. This makes the exterior bundle $\bigwedge^ \bullet(TZ) \oplus \bigwedge^ \bullet(T^\ast Z)$ into the graded analogue of $TZ \oplus T^\ast Z$ equipped with the usual Courant bracket (see \cite{Courant1990}).  As the latter is the standard example of a Courant algebroid, we can consider $\bigwedge^ \bullet(TZ) \oplus \bigwedge^ \bullet(T^\ast Z)$ as a ``graded Courant algebroid.''

\paragraph{Multivectors and Forms.}

Let $Z$ be an arbitrary manifold.  We denote the $k$-th exterior power of $TZ$ and $T^\ast Z$ by $\bigwedge^k(TZ)$ and $\bigwedge^k(T^\ast Z)$, respectively. We consider a form-valued pairing between $\Sigma\in\bigwedge^k(T^\ast Z)$ and $\Gamma\in\bigwedge^l(TZ)$, given by 
\begin{equation} \label{pairing1}
	\left< \Sigma, \Gamma \right> = \mathbf{i}_\Gamma \Sigma \in \mbox{$\bigwedge$}^{k-l}(T^\ast Z),
\end{equation}
for $k\geq l$, and $\left< \Sigma, \Gamma \right>=0$ for $k<l$.  For fixed $n \le \dim Z$, we introduce the product bundles 
\begin{equation} \label{Dr}
	L_r := \mbox{$\bigwedge$}^r(TZ) \times_Z \mbox{$\bigwedge$}^{n + 1 - r}(T^\ast Z), 
\end{equation}
where $r = 1, \ldots, n$.  Using the pairing \eqref{pairing1}, we can then define  graded commutative and  graded anticommutative pairings between the elements of $L_r$ and $L_s$ (for $r, s = 1, \ldots, n$) as follows.  For $(\Gamma, \Sigma) \in L_r$ and $(\Gamma', \Sigma') \in L_s$, we put 
\begin{align} \label{antisymm}
	\left<\!\left< (\Gamma, \Sigma), (\Gamma', \Sigma') \right>\!\right>_-
	:=&  \frac{1}{2} \left( \left<\Sigma,\Gamma'\right> - (-1)^{rs} \left<\Sigma',\Gamma\right> \right)\notag\\
	=&  \frac{1}{2} \left( \mathbf{i}_{\Gamma'} \Sigma - (-1)^{rs} \mathbf{i}_{\Gamma} \Sigma' \right),
\end{align}
and
\begin{align} \label{symm}
	\left<\!\left< (\Gamma, \Sigma), (\Gamma', \Sigma') \right>\!\right>_+
	:=&  \frac{1}{2} \left( \left<\Sigma,\Gamma'\right> + (-1)^{rs} \left<\Sigma',\Gamma\right> \right)\notag\\
	=&  \frac{1}{2} \left(
		\mathbf{i}_{\Gamma'} \Sigma + (-1)^{rs} \mathbf{i}_\Gamma \Sigma'
		\right).
\end{align}
Note that both pairings take values in $\bigwedge^{n + 1 - r - s}(T^\ast Z)$.  As a consequence, both pairings are trivially zero whenever $r + s > n + 1$.

There is a slight difference between our notation and the one used by \cite{Courant1990}.  Here, the subscript `+' and `-' in the pairing reflects the fact that the pairing is graded commutative and graded anti-commutative, respectively, but this is opposite to the notational convention of Courant.

\paragraph{The Multi-Courant Bracket and Multiplication on the Space of Sections.}  

There exists a wedge product between sections of $L_r$ and $L_s$, defined in the following way.  For $(\Gamma, \Sigma) \in L_r$ and $(\Gamma', \Sigma') \in L_s$, with $r + s \le n$, we define
\begin{align}
	(\Gamma, \Sigma) \wedge (\Gamma', \Sigma') & := 
	\left(\Gamma \wedge \Gamma', \pairleft (\Gamma, \Sigma), (\Gamma', \Sigma') \pairright_+ \right) \nonumber \\
	& \phantom{:}= \left( \Gamma \wedge \Gamma', \frac{1}{2} \left( \mathbf{i}_{\Gamma'} \Sigma + (-1)^{rs} \mathbf{i}_\Gamma \Sigma' \right) \right).
		\label{mult}
\end{align}
Note that $(\Gamma, \Sigma) \wedge (\Gamma', \Sigma') \in L_{r + s}$.   If $r + s > n$, we let $(\Gamma, \Sigma) \wedge (\Gamma', \Sigma') = 0$.

Furthermore, there exists a canonical bracket on the space of sections of $L$ which is reminiscent of the usual Courant bracket.  This multi-Courant bracket is defined by means of the following bracket on the homogeneous sections.  On sections of $L_r \times L_s$, define the multi-Courant bracket  $\bracleft \cdot, \cdot \bracright : L_r \times L_s \rightarrow L_{r + s - 1}$ by
\begin{multline} \label{courant}
	\bracleft (\Gamma, \Sigma),  (\Gamma', \Sigma') \bracright  \\
	 := 
	\left( [\Gamma, \Gamma'], (-1)^{(r - 1)s} \pounds_{\Gamma} \Sigma'
		+ (-1)^s \pounds_{\Gamma'} \Sigma 
		- (-1)^s \mathbf{d} 
		\left<\!\left< (\Gamma, \Sigma), (\Gamma', \Sigma') 
			\right>\!\right>_+ \right) 
\end{multline}
when $r + s \le n + 1$.  When $r + s > n + 1$, we define the bracket to be zero.
In the above, $[\Gamma, \Gamma']$ is the Schouten-Nijenhuis bracket of the multivector fields $\Gamma$ and $\Gamma'$, and the Lie derivatives denote the generalized Lie derivatives of a form along a multivector field.  Note that \eqref{courant} takes values in $L_{r + s - 1}$.   Some properties of the Schouten-Nijenhuis bracket and the generalized Lie derivative are described in Appendix~\ref{sec:lie}.

\paragraph{Automorphisms of the Multi-Courant Bracket.}  

Just like in the case of the standard Courant bracket, there exists a distinguished class of automorphisms of the multi-Courant bracket, defined by means of closed forms as follows. For every closed $(n + 1)$-form $\sigma$, define the mapping $\Phi_\sigma: L_r \rightarrow L_r$ by
\begin{equation} \label{automorphism}
	\Phi_\sigma (\Gamma, \Sigma) = (\Gamma, \Sigma + \mathbf{i}_\Gamma \sigma).
\end{equation}
For notational convenience, we let $L := \oplus_{r = 1}^n L_r$ and we denote the map induced by \eqref{automorphism} on $L$ by $\Phi_\sigma: L \to L$.

\begin{proposition}  For every closed $(n + 1)$-form $\sigma$, the associated map $\Phi_\sigma: L \rightarrow L$  is an automorphism of the multiplication \eqref{mult} and the multi-Courant bracket \eqref{courant}:
\begin{equation} \label{autcourant}
	\Phi_\sigma \left( \bracleft (\Gamma, \Sigma), (\Gamma', \Sigma') \bracright \right) = 
		\bracleft \Phi_\sigma (\Gamma, \Sigma), \Phi_\sigma (\Gamma', \Sigma') \bracright
\end{equation}
and 
\begin{equation} \label{autmult}
	\Phi_\sigma \left( (\Gamma, \Sigma) \wedge  (\Gamma', \Sigma') \right) 
	= \Phi_\sigma  (\Gamma, \Sigma) \wedge  \Phi_\sigma (\Gamma', \Sigma') 
\end{equation}
for all $(\Gamma, \Sigma), (\Gamma', \Sigma') \in L$.
\end{proposition}

\begin{proof}
We have that the right-hand side of \eqref{autcourant} is given by 
\begin{multline*}
\bracleft \Phi_\sigma (\Gamma, \Sigma), \Phi_\sigma (\Gamma', \Sigma') \bracright 
	 = \\
	 \bracleft  (\Gamma, \Sigma),  (\Gamma', \Sigma') \bracright 
	 + \Big( 0, (-1)^{(r-1)s} \lie_\Gamma \mathbf{i}_{\Gamma'} \sigma 
		+(-1)^{s} \lie_{\Gamma'} \mathbf{i}_\Gamma \sigma
		- (-1)^s \mathbf{d} \mathbf{i}_{\Gamma'} \mathbf{i}_{\Gamma} \sigma\Big)
\end{multline*}
The terms involving $\sigma$ can be rewritten as 
\begin{align*}
(-1)^{(r-1)s} \lie_\Gamma \mathbf{i}_{\Gamma'} \sigma 
		& +(-1)^{s} \lie_{\Gamma'} \mathbf{i}_\Gamma \sigma
		- (-1)^s \mathbf{d} \mathbf{i}_{\Gamma'} \mathbf{i}_{\Gamma} \sigma\\
	& = 
	(-1)^{(r-1)s} \lie_\Gamma \mathbf{i}_{\Gamma'} \sigma
	- \mathbf{i}_{\Gamma'} \pounds_\Gamma \sigma 
	- (-1)^r \mathbf{i}_{\Gamma'} \mathbf{i}_\Gamma \mathbf{d} \sigma \\
	& = \mathbf{i}_{[\Gamma, \Gamma']} \sigma - 
		(-1)^r \mathbf{i}_{\Gamma \wedge \Gamma'} \mathbf{d}  \sigma,
\end{align*}
where we have used the Koszul identity \eqref{koszulid} in the last step.  Since $\mathbf{d}  \sigma = 0$, this shows us that 
\begin{align*}
	\bracleft \Phi_\sigma (\Gamma, \Sigma), \Phi_\sigma (\Gamma', \Sigma') \bracright 
	& = \bracleft  (\Gamma, \Sigma),  (\Gamma', \Sigma') \bracright +
		(0, \mathbf{i}_{[\Gamma, \Gamma']} \sigma) \\
	& = \Phi_\sigma \left( \bracleft (\Gamma, \Sigma), (\Gamma', \Sigma') \bracright \right).
\end{align*}

For the proof of \eqref{autmult}, we observe that 
\begin{align*}
\Phi_\sigma \left( (\Gamma, \Sigma) \wedge  (\Gamma', \Sigma') \right) & = 
	\Phi_\sigma\left( \Gamma \wedge \Gamma',  \mathbf{i}_{\Gamma'} \Sigma \right) \\
	& = \left(\Gamma \wedge \Gamma', \frac{1}{2} \left(\mathbf{i}_{\Gamma'} \Sigma 
		+ (-1)^{rs} \mathbf{i}_\Gamma \Sigma' \right)
	+ \mathbf{i}_{\Gamma \wedge \Gamma'} \sigma \right),
\end{align*}
while 
\begin{align*}
\Phi_\sigma  (\Gamma, \Sigma) \wedge  \Phi_\sigma (\Gamma', \Sigma') & = 
	(\Gamma, \Sigma + \mathbf{i}_\Gamma \sigma) \wedge 
		(\Gamma', \Sigma' + \mathbf{i}_{\Gamma'} \sigma) \\
		& = \left(\Gamma \wedge \Gamma', \frac{1}{2} \big[\mathbf{i}_{\Gamma'}( \Sigma + \mathbf{i}_\Gamma \sigma) 
			+ (-1)^{rs} \mathbf{i}_{\Gamma}( \Sigma' + \mathbf{i}_{\Gamma'} \sigma) \big]
		\right)
\end{align*}
so that \eqref{autmult} holds without any conditions on $\sigma$.
\end{proof}

\section{Geometry of Multi-Dirac Structures}

In this section, we define a multi-Dirac structure on a manifold $Z$ as a subbundle of $\bigwedge^ \bullet(TZ) \oplus \bigwedge^ \bullet(T^\ast Z)$ which is maximally isotropic with respect to the graded antisymmetric pairing $\left<\!\left< \cdot, \cdot \right>\!\right>_-$ introduced in \eqref{antisymm}.  After defining a notion of integrability for multi-Dirac structures, we then prove the main result of this paper (Theorem~\ref{thm:gerstenhaber}): the space of sections of an integrable multi-Dirac structure can be equipped with the structure of a Gerstenhaber algebra.

\begin{definition}
Let $V_r$ be a subbundle of $L_r$. The \textbf{$s$-orthogonal complement} of $V_r$ is the subbundle 
$(V_r)^{\perp, s}$ of $L_s$ which consists of all $(\Gamma', \Sigma') \in L_s$ such that 
\[
	\left<\!\left< (\Gamma, \Sigma), (\Gamma', \Sigma') \right>\!\right>_- = 0 
	\quad \text{for all $(\Gamma, \Sigma) \in V_r$}.
\]
\end{definition}
Note that $(V_r)^{\perp, s} \subset L_s$ for $r + s \le n + 1$ and $(V_r)^{\perp, s} = L_s$ whenever $r + s > n + 1$, irrespective of $V_r$.

  We can now introduce graded multi-Dirac structures in an analogous way to the original definition of standard Dirac structures by \cite{Courant1990}.

\begin{definition}
A {\bfi multi-Dirac structure of degree $n$ on $Z$} is a sequence of subbundles $D_r \subset L_r$, $r = 1, \ldots n$ satisfying the \emph{maximally $s$-isotropic property}:
\begin{equation} \label{isotropy}
	(D_r)^{\perp, s} = D_s 	
\end{equation}
for $r, s$ such that $r + s \le n + 1$.  We set $D := D_1 \oplus \ldots \oplus D_{n}$, so that $D \subset \bigwedge^ \bullet(TZ) \oplus \bigwedge^ \bullet(T^\ast Z)$ and we say that $D$ is \emph{maximally isotropic}.
\end{definition}

Multi-Dirac structures have a number of properties which are reminiscent of standard Dirac structures.  The remainder of this section is devoted to these properties.

\paragraph{Integrability of Multi-Dirac Structures.}  

A special class of Dirac structures consists of those for which the space of sections is closed under the wedge product and the multi-Courant bracket.  Note first that for a multi-Dirac structure, the wedge product and multi-Courant bracket take on slightly simpler forms:
\begin{equation} \label{dirmult}
	(\Gamma, \Sigma) \wedge (\Gamma', \Sigma') = ( \Gamma \wedge \Gamma',  \mathbf{i}_{\Gamma'} \Sigma )
\end{equation}
and 
\begin{equation} \label{dircourant}
	\bracleft (\Gamma, \Sigma),  (\Gamma', \Sigma') \bracright  
	 = 
	\left( [\Gamma, \Gamma'], (-1)^{(r-1)s}\lie_\Gamma \Sigma' - 
		\mathbf{i}_{\Gamma'} \mathbf{d} \Sigma
	 \right)
\end{equation}
for all $(\Gamma, \Sigma) \in D_r$ and $(\Gamma', \Sigma') \in D_s$, where $r, s = 1, \ldots, n$.  Now, it is easy to show that the wedge product always takes values in $D$: with the notations of above, and with $(\Gamma'', \Sigma'') \in D_t$, we have 
\begin{align*}
\left<\!\left<(\Gamma, \Sigma) \wedge (\Gamma', \Sigma'), (\Gamma'', \Sigma'') \right>\!\right>_-
	& =  \frac{1}{2} \left( \mathbf{i}_{\Gamma''} \mathbf{i}_{\Gamma'} \Sigma - (-1)^{(r+s)t} \mathbf{i}_{\Gamma \wedge \Gamma'} \Sigma'' \right)
	\\
	& =  \frac{(-1)^{st}}{2} \mathbf{i}_{\Gamma'} \left(  \mathbf{i}_{\Gamma''} \Sigma - (-1)^{rt} \mathbf{i}_{\Gamma} \Sigma'' \right) \\
	& = (-1)^{st} \mathbf{i}_{\Gamma'}
	\left<\!\left<(\Gamma, \Sigma), (\Gamma'', \Sigma'') \right>\!\right>_-,
\end{align*}
which is identically zero.  Hence, the wedge product of two sections of $D$ is always again a section of $D$.  In contrast, the multi-Courant bracket of two sections of $D$ is not necessarily again a section of $D$.  When this is nevertheless the case, we will say that $D$ is an integrable multi-Dirac structure.

\begin{definition}
	A multi-Dirac structure $D$ of degree $n$ is said to be \textbf{integrable} if the space of sections of $D$ is closed under the multi-Courant bracket \eqref{courant}.  That is, for all $(\Gamma, \Sigma) \in D_r$ and $(\Gamma', \Sigma') \in D_s$ (where $r, s = 1, \ldots, n$) the following holds:
	\[
		\bracleft (\Gamma, \Sigma),  (\Gamma', \Sigma') \bracright \in D_{r + s - 1}.
	\]
\end{definition}

The integrability of multi-Dirac structures can be determined by the vanishing of a certain {\bfi integrability tensor}, which we now define.  First of all, consider a multi-Dirac structure $D = D_1 \oplus \cdots \oplus D_n$ of degree $n$ and let $(\Gamma, \Sigma) \in D_r$, $(\Gamma', \Sigma') \in D_s$, and $(\Gamma'', \Sigma'') \in D_t$ be arbitrary homogeneous sections of $D$ and consider the following map: 
\begin{equation} \label{inttensor}
	T_D( (\Gamma, \Sigma), (\Gamma', \Sigma'), (\Gamma'', \Sigma'') )
	=  2\left<\!\left<(\Gamma, \Sigma), \bracleft (\Gamma', \Sigma'), (\Gamma'', \Sigma'') \bracright \right>\!\right>_-.
\end{equation}
In appendix~\ref{app:tensor}, we show that $T_D$ is a tensor.  

By definition, $D$ is integrable if and only if $T_D$ vanishes identically.  In the following lemma, we give a clearer expression for $T_D$.

\begin{lemma}\label{T_D_expression}
	The integrability tensor $T_D$ can be written as 
	\begin{multline} \label{integrability}
		T_D((\Gamma, \Sigma), (\Gamma', \Sigma'), (\Gamma'', \Sigma''))
		=
		(-1)^{t(r-1)} \Big[
			- (-1)^{(r + s)t} \mathbf{d} \mathbf{i}_{\Gamma'} 
				\mathbf{i}_{\Gamma''} \Sigma \\ 
		+ (-1)^{s(t - 1)} \mathbf{i}_{\Gamma'} \pounds_\Gamma \Sigma'' + (-1)^{r(s - 1)} \mathbf{i}_{\Gamma} \pounds_{\Gamma''} \Sigma'
+ (-1)^{t(r - 1)} \mathbf{i}_{\Gamma''} \pounds_{\Gamma'} \Sigma
		\Big]
	\end{multline}
	where $(\Gamma, \Sigma) \in D_r$, $(\Gamma', \Sigma') \in D_s$, and $(\Gamma'', \Sigma'') \in D_t$.
\end{lemma}

\begin{proof}
The expression \eqref{inttensor} is given in full by 
\[
	T_D =   (-1)^{(s-1)t} \underbrace{\mathbf{i}_{\Gamma}
	\pounds_{\Gamma'} \Sigma''}_{\text{(A)}}
	- \underbrace{\mathbf{i}_{\Gamma} 
	\mathbf{i}_{\Gamma''} \mathbf{d} \Sigma'}_{\text{(B)}}
	- (-1)^{r(s + t - 1)} \underbrace{\mathbf{i}_{[\Gamma', \Gamma'']} 
		\Sigma}_{\text{(C)}},
\]
where we have suppressed the arguments of $T_D$ for the sake of clarity.  We now tackle terms (A), (B), and (C) individually.  

Term (A) can be rewritten as 
\begin{align*}
\mathbf{i}_{\Gamma}
	\pounds_{\Gamma'} \Sigma'' & = 
		\mathbf{i}_{\Gamma} \mathbf{d} \mathbf{i}_{\Gamma'} \Sigma''
		- (-1)^s \mathbf{i}_{\Gamma} \mathbf{i}_{\Gamma'} \mathbf{d} \Sigma'' \\
		& = 
		\mathbf{i}_{\Gamma} \mathbf{d} \mathbf{i}_{\Gamma'} \Sigma''
		- (-1)^{rs + s} \mathbf{i}_{\Gamma'} \mathbf{i}_{\Gamma} \mathbf{d} \Sigma'' \\
		& = \mathbf{i}_{\Gamma} \mathbf{d} \mathbf{i}_{\Gamma'} \Sigma''
		- (-1)^{(r - 1)(s - 1)} \left( \mathbf{i}_{\Gamma'} \pounds_\Gamma \Sigma'' - \mathbf{i}_{\Gamma'} \mathbf{d} \mathbf{i}_{\Gamma} \Sigma'' \right). 
\end{align*}
For term (B), we have 
\[
\mathbf{i}_{\Gamma} 
	\mathbf{i}_{\Gamma''} \mathbf{d} \Sigma' = (-1)^t \mathbf{i}_{\Gamma}
	\left( \mathbf{d} \mathbf{i}_{\Gamma''} \Sigma' 
		- \pounds_{\Gamma''} \Sigma' \right),
\]
while term (C) is given by 
\begin{align*}
\mathbf{i}_{[\Gamma', \Gamma'']}  \Sigma
& = (-1)^{(s - 1)t} \pounds_{\Gamma'} \mathbf{i}_{\Gamma''} \Sigma
- \mathbf{i}_{\Gamma''} \pounds_{\Gamma'} \Sigma \\
& = (-1)^{(s - 1)t} \left( \mathbf{d} \mathbf{i}_{\Gamma'} \mathbf{i}_{\Gamma''} \Sigma - (-1)^s \mathbf{i}_{\Gamma'} \mathbf{d}
\mathbf{i}_{\Gamma''} \Sigma \right) 
- \mathbf{i}_{\Gamma''} \pounds_{\Gamma'} \Sigma.
\end{align*}
Collecting all of these expressions together, we have that 
\begin{multline} \label{intermediatetd}
	T_D = 
	- (-1)^{(s - 1)t + (r - 1)(s - 1)}  \mathbf{i}_{\Gamma'} \pounds_\Gamma \Sigma'' + (-1)^t \mathbf{i}_{\Gamma}\pounds_{\Gamma''} \Sigma' \\
	- (-1)^{r(s + t - 1)} \left( 
	(-1)^{(s - 1)t} \mathbf{d} \mathbf{i}_{\Gamma'} \mathbf{i}_{\Gamma''} \Sigma -
	\mathbf{i}_{\Gamma''} \pounds_{\Gamma'} \Sigma
	\right),
\end{multline}
where we have used the fact that
\begin{align*}
	\mathbf{i}_{\Gamma'} \Sigma'' - 
		(-1)^{st} \mathbf{i}_{\Gamma''} \Sigma' &=2\left<\!\left<(\Gamma'',\Sigma''),(\Gamma',\Sigma')\right>\!\right>= 0\\
\intertext{and} 
	\mathbf{i}_{\Gamma} \Sigma'' - 
		(-1)^{rt} \mathbf{i}_{\Gamma''} \Sigma &=2\left<\!\left<(\Gamma'',\Sigma''),(\Gamma,\Sigma)\right>\!\right>= 0,
\end{align*}
since $(\Gamma, \Sigma), (\Gamma', \Sigma')$, and $(\Gamma'', \Sigma'')$ are elements of $D$.  After rearranging the signs in \eqref{intermediatetd}, we obtain \eqref{integrability}.
\end{proof}

We now arrive at the main result of this section: the space of sections of an integrable multi-Dirac structure forms a Gerstenhaber algebra (for the definition of a Gerstenhaber algebra, see \cite{Ge1963, Kanatchikov1995, CaWe1999}).

\begin{theorem} \label{thm:gerstenhaber}
	Let $D$ be an integrable multi-Dirac structure of order $n$.  The space of sections of $D$ forms a \textbf{Gerstenhaber algebra}: for all $(\Gamma, \Sigma) \in D_r$, $(\Gamma', \Sigma') \in D_s$, and $(\Gamma'', \Sigma'') \in D_t$, we have
	\begin{itemize}
		\item Graded anti-commutativity: $\bracleft (\Gamma, \Sigma), (\Gamma', \Sigma') \bracright = - (-1)^{(r-1)(s-1)} \bracleft (\Gamma', \Sigma'), (\Gamma, \Sigma) \bracright$;
		\item Graded Leibniz identity: 
		\begin{multline*}
			\bracleft (\Gamma, \Sigma), 
				(\Gamma', \Sigma') \wedge (\Gamma'', \Sigma'')
			\bracright = \\
			\bracleft (\Gamma, \Sigma), 
				(\Gamma', \Sigma') \bracright \wedge (\Gamma'', \Sigma'')
			+ (-1)^{(r - 1)s} (\Gamma', \Sigma') \wedge
			\bracleft (\Gamma, \Sigma), 
				(\Gamma'', \Sigma'') \bracright;
		\end{multline*}
	\item Graded Jacobi identity: 
	\begin{multline} \label{gradj}
		\bracleft (\Gamma, \Sigma), 
		\bracleft (\Gamma', \Sigma'), 
			(\Gamma'', \Sigma'') \bracright \bracright
		+ (-1)^{(t-1)(r + s)} 
		\bracleft (\Gamma'', \Sigma''), 
		\bracleft (\Gamma, \Sigma), 
			(\Gamma', \Sigma') \bracright \bracright \\
		+ (-1)^{(r - 1)(s + t)} 
				\bracleft (\Gamma', \Sigma'), 
		\bracleft (\Gamma'', \Sigma''), 
			(\Gamma, \Sigma) \bracright \bracright = 0
	\end{multline}
	\end{itemize}
\end{theorem}

\begin{proof}
The proofs of the first two properties are tedious but straightforward verifications, which we leave to the reader.  For the proof of the Jacobi identity, we proceed as follows.  We first introduce the {\bfi Jacobiator} $\mathcal{J}((\Gamma, \Sigma), (\Gamma', \Sigma'), (\Gamma'', \Sigma''))$ as the left-hand side of \eqref{gradj}.  Our goal is to show that $\mathcal{J} = 0$ when restricted to $D$.

With the notations of the theorem, we have that 
\begin{multline*}
\bracleft (\Gamma, \Sigma), 
		\bracleft (\Gamma', \Sigma'), 
			(\Gamma'', \Sigma'') \bracright \bracright =
	\\ \left( [\Gamma, [\Gamma', \Gamma'']],  
		(-1)^{(r - 1)(s + t - 1)} 
		\pounds_{\Gamma} \left( (-1)^{(s - 1)t} \pounds_{\Gamma'} \Sigma
		- \mathbf{i}_{\Gamma''} \mathbf{d} \Sigma' \right) 
		- \mathbf{i}_{[\Gamma', \Gamma'']} \mathbf{d} \Sigma \right)
\end{multline*}
and similar expressions for the other terms in the Jacobi identity can be obtained through cyclic permutation.  We now collect all the terms that involve $\Sigma$, finding
\[
	- \mathbf{i}_{[\Gamma', \Gamma'']} \mathbf{d} \Sigma
	+ (-1)^t \pounds_{\Gamma''} \mathbf{i}_{\Gamma'} \mathbf{d} \Sigma
	+ (-1)^{st + 1} \pounds_{\Gamma'} \pounds_{\Gamma''} \Sigma,
\]
which can be rewritten as 
\[
	(-1)^{s + t} \left( 
		\mathbf{d} \mathbf{i}_{[\Gamma', \Gamma'']} \Sigma 
		+ \pounds_{\Gamma''} \mathbf{d} \mathbf{i}_{\Gamma'} \Sigma 
	\right).
\]
Putting everything together, the left-hand side of \eqref{gradj} can then be written as 
\begin{multline} \label{jacobiator}
	\mathcal{J}((\Gamma, \Sigma), (\Gamma', \Sigma'), (\Gamma'', \Sigma''))
	=	(-1)^{s + t} \left( 
		\mathbf{d} \mathbf{i}_{[\Gamma', \Gamma'']} \Sigma 
		+ \pounds_{\Gamma''} \mathbf{d} \mathbf{i}_{\Gamma'} \Sigma 
	\right) \\
	+ (-1)^{(r - 1)(s + t) + r + t} 
	\left(  \mathbf{d} \mathbf{i}_{[\Gamma'', \Gamma]} \Sigma' 
		+ \pounds_{\Gamma} \mathbf{d} \mathbf{i}_{\Gamma''} \Sigma'  \right) 
	+ (-1)^{t(r + s)} 
		\left(  \mathbf{d} \mathbf{i}_{[\Gamma, \Gamma']} \Sigma'' 
		+ \pounds_{\Gamma'} \mathbf{d} \mathbf{i}_{\Gamma} \Sigma''  \right).
\end{multline}
Now we rewrite the terms involving Schouten-Nijenhuis brackets using the Koszul identity \eqref{koszulid} as follows: 
\begin{align*}
	\mathbf{d} \mathbf{i}_{[\Gamma', \Gamma'']} \Sigma
	& = (-1)^{(s - 1)(t - 1)} \pounds_{\Gamma'} \mathbf{d} \mathbf{i}_{\Gamma''} \Sigma - \mathbf{d} \mathbf{i}_{\Gamma''} \pounds_{\Gamma'} \Sigma \\
	& = (-1)^{(s - 1)(t - 1) + rt}
	\pounds_{\Gamma'} \mathbf{d} \mathbf{i}_{\Gamma} \Sigma'' - \mathbf{d} \mathbf{i}_{\Gamma''} \pounds_{\Gamma'} \Sigma, 
\end{align*}
and similarly for the other terms.  Note that the first term on the right hand side cancels exactly with the last term in \eqref{jacobiator}.  Similar cancellations occur for the other terms as well, and we are eventually left with 
\begin{align*}
	&\mathcal{J}((\Gamma, \Sigma), (\Gamma', \Sigma'), (\Gamma'', \Sigma''))
	\\
	&\quad= - (-1)^{s + rt} \mathbf{d} 
	\left( 
		(-1)^{(r - 1)t} \mathbf{i}_{\Gamma''} \pounds_{\Gamma'} \Sigma
		+ (-1)^{(s - 1)r} \mathbf{i}_{\Gamma} \pounds_{\Gamma''} \Sigma'
		+ (-1)^{(t - 1)s} \mathbf{i}_{\Gamma'} \pounds_{\Gamma} \Sigma''
	\right)\\
	&\quad= - (-1)^{rs + rt + r + s  + t} \mathbf{d} \left( 
	T_D((\Gamma, \Sigma), (\Gamma', \Sigma'), (\Gamma'', \Sigma''))
	\right),
\end{align*}
where we used Lemma~\ref{T_D_expression} and the fact that $\mathbf{d}^2=0$. Hence, the Jacobiator is simply the exterior derivative of the integrability tensor, and the Jacobi identity is satisfied whenever the multi-Dirac structure is integrable, i.e., when $T_D \equiv 0$.
\end{proof}

Finally, let $D = D_1 \oplus \cdots \oplus D_n$ be an integrable multi-Dirac structure on a manifold $Z$.  The sections of $D$ form a Gerstenhaber algebra, but are also ``anchored'' to the base manifold $Z$ in a special way.  To elucidate this, we recall that the exterior algebra $\bigwedge^ \bullet (TZ)$ of multi-vector fields, equipped with the wedge product and the Schouten-Nijenhuis bracket, is the standard example of a Gerstenhaber algebra (see \cite{Marle1997}).  It is not hard to show that the projection map $\rho : D \to \bigwedge^ \bullet (TZ)$, given by 
\[
	\rho((\Gamma_1, \Sigma_1), \ldots, (\Gamma_n, \Sigma_n) )
	= ( \Gamma_1, \ldots, \Gamma_n)
\]
is a homomorphism of Gerstenhaber algebras.  In most cases, this map is neither injective nor surjective.

\section{Examples}

\subsection{Dirac Structures and Higher-Order Dirac Structures} \label{sec:hodirac}

\paragraph{Dirac Structures.}  Let $D$ be a multi-Dirac structure of degree one on a manifold $Z$.  We claim that $D$ is a standard Dirac structure in the sense of \cite{Courant1990}.  Indeed, $D$ is by definition a subbundle of $TZ \oplus T^\ast Z$, and the graded anti-commutative pairing \eqref{antisymm} becomes in this case
\[
	\left<\!\left< (\Gamma, \Sigma), (\Gamma', \Sigma') \right>\!\right>_-
	=  \frac{1}{2} \left( \mathbf{i}_{\Gamma'} \Sigma + \mathbf{i}_{\Gamma} \Sigma' \right),
\]
where $\Gamma, \Gamma \in TZ$ and $\Sigma, \Sigma' \in T^\ast Z$.  The requirement that $D$ is maximally isotropic under this pairing then makes $D$ into a Dirac structure.

It is also instructive to consider the multi-Courant bracket in the degree-one case.  Putting $r = s = 1$ in \eqref{courant}, we have that the bracket on $TZ \oplus T^\ast Z$ is given by 
\[
	\bracleft (\Gamma, \Sigma),  (\Gamma', \Sigma') \bracright 
	 = 
	\left( [\Gamma, \Gamma'],  \pounds_{\Gamma} \Sigma'
		- \pounds_{\Gamma'} \Sigma 
		- \frac{1}{2} \mathbf{d} 
		\left( \mathbf{i}_{\Gamma'} \Sigma - 
			\mathbf{i}_\Gamma \Sigma' \right) \right), 
\]
which is precisely the standard bracket of \cite{Courant1990}.  We conclude that Dirac structures are a special case of multi-Dirac structures.

\paragraph{Higher-Order Dirac Structures.}  Dirac structures in the sense of \cite{Courant1990} fit into a hierarchy of higher-order structures, which were described by \cite{Hitchin2003}, \cite{Gualtieri2004} and \cite{Zambon2010}; see also the references therein.  We now show that every multi-Dirac structure induces a higher-order Dirac structure.  Under some regularity assumptions, \cite{Zambon2010} has shown that there is in fact a one-to-one correspondence between multi-Dirac structures and higher-order Dirac structures.

Let $D$ be a multi-Dirac structure of degree $n$ on $Z$, and consider the component $D_{1} \subset L_1 = TZ \oplus \bigwedge^n (T^\ast Z)$.  We claim that $D_1$ is in itself a higher-order Dirac structure of degree $n$.  This follows when we examine the pairing \eqref{antisymm}, restricted to elements of $L_1$:
\[
	\left<\!\left< (\Gamma, \Sigma), (\Gamma', \Sigma') \right>\!\right>_-
	=  \frac{1}{2} \left( \mathbf{i}_{\Gamma'} \Sigma + \mathbf{i}_{\Gamma} \Sigma' \right),
\]
where now $\Gamma, \Gamma' \in TZ$ and $\Sigma, \Sigma' \in \bigwedge^n(Z)$.  As a special case of the maximally isotropic property \eqref{isotropy}, we then have that 
\[
	(D_1)^{\perp, 1} = D_1, 
\]
which makes $D_1$ into a higher-order Dirac structure of degree $n$.  Restricting the multi-Courant bracket to sections of $L_1$, we have that the bracket is again given by 
\[
	\bracleft (\Gamma, \Sigma),  (\Gamma', \Sigma') \bracright 
	 = 
	\left( [\Gamma, \Gamma'],  \pounds_{\Gamma} \Sigma'
		- \pounds_{\Gamma'} \Sigma 
		- \frac{1}{2} \mathbf{d} 
		\left( \mathbf{i}_{\Gamma'} \Sigma - 
			\mathbf{i}_\Gamma \Sigma' \right) \right), 
\]
which is nothing but the Courant bracket on sections of $L_1$.

\paragraph{Links with Lie Algebroids.}  Let $D = D_1 \oplus \cdots D_n$ be an integrable multi-Dirac structure on a manifold $Z$.   We have just demonstrated that $D_1$ is a higher-order Dirac structure.  As a result, $D_1$ is a Lie algebroid over $Z$ with bracket the multi-Courant bracket restricted to $D_1$ and anchor the projection $\rho : D_1 \to TZ$ onto the first factor.  The image $\rho(D_1)$ of $\rho$ forms an integrable singular distribution on $Z$ so that $Z$ is foliated by integral submanifolds.  It can now be shown that these leaves are pre-multisymplectic manifolds, that is, they are equipped with a form $\Omega_D$ of degree $n$ which is closed but not necessarily nondegenerate.\footnote{This observation, and the construction of $\Omega_D$, are due to Eduardo Mart\'{\i}nez.}  The form $\Omega_D$ is defined as follows: let $(v_1, \alpha_1), \ldots, (v_n, \alpha_n)$ be sections of $D_1$, and define
\[
	\Omega_D(v_1, \ldots, v_n) := \alpha_n(v_1, \ldots, v_{n-1}).
\]
Because of isotropy, $\Omega_D$ is totally antisymmetric in all of its arguments.

\subsection{The Graph of a Differential Form}

In this section, we show that the graph (in some suitable sense) of a differential form determines a multi-Dirac structure.   This multi-Dirac structure was already defined in \cite{VaYoLeMa2011}. We refer to that paper for an overview of applications in classical field theory, and for proofs of the basic theorems.  In this paper, we investigate the geometric aspects of the multi-Dirac structure in further detail: in particular, we show that the resulting multi-Dirac structure is integrable if and only if the underlying form is closed.

Let $\Omega$ be a form of degree $n + 1$.  We define the graph of $\Omega$ as the sequence of subspaces $D_r \subset L_r$ given by 
\begin{equation} \label{canondirac}
	D_r := \left\{ (\Gamma, \mathbf{i}_\Gamma \Omega)\in L_r\,\,|\,\,\Gamma \in \mbox{$\bigwedge$}^r(TZ) \right\}
\end{equation}
 for $r = 1, \ldots, n$.

\begin{proposition}   \label{prop:canondirac}
The sequence \eqref{canondirac} of bundles $D_r$ satisfies the maximally $s$-isotropic property \eqref{isotropy}:  whenever $r, s$ are such that $r + s \le n + 1$, the following holds:
\[
	(D_r)^{\perp, s} = D_s. 
\]	
Hence, the direct sum $D = D_1 \oplus \cdots \oplus D_n$ is a multi-Dirac structure of degree $n$.
\end{proposition}
\begin{proof} See \cite{VaYoLeMa2011}.
%
%
\end{proof}

We now turn to the question whether this multi-Dirac structure is integrable.

\begin{proposition}
	The canonical multi-Dirac structure \eqref{canondirac} is integrable if and only if the underlying form $\Omega$ is closed: $\mathbf{d} \Omega = 0$.
\end{proposition}
\begin{proof}
The bracket is closed if, for arbitrary sections $(\Gamma, \mathbf{i}_\Gamma \Omega) \in D_r$ and $(\Gamma', \mathbf{i}_{\Gamma'} \Omega) \in D_s$, we have that $\bracleft (\Gamma, \mathbf{i}_\Gamma \Omega), (\Gamma', \mathbf{i}_{\Gamma'} \Omega) \bracright \in D_{r + s - 1}$ or 
\[
	\left( [\Gamma, \Gamma'], (-1)^{(r - 1)s} \pounds_\Gamma \Sigma' 
		- \mathbf{i}_{\Gamma'} \mathbf{d} \Sigma \right)
	\in D_{r + s - 1}
\]
with $\Sigma = \mathbf{i}_\Gamma \Omega$ and $\Sigma' = \mathbf{i}_{\Gamma'} \Omega$.  This is the case if 
\begin{equation} \label{eq1proof}
	\mathbf{i}_{[\Gamma, \Gamma']} \Omega = (-1)^{(r-1)s}\lie_\Gamma \mathbf{i}_{\Gamma'} \Omega - 
	\mathbf{i}_{\Gamma'} \mathbf{d} \mathbf{i}_\Gamma \Omega.	
\end{equation}

The left-hand side of \eqref{eq1proof} can be rewritten using the Koszul identity \eqref{koszulid} as
\begin{align*}
	\mathbf{i}_{[\Gamma, \Gamma']} \Omega & =
	(-1)^{(r - 1)s} \pounds_\Gamma \mathbf{i}_{\Gamma'} \Omega
	- \mathbf{i}_{\Gamma'} \pounds_\Gamma \Omega \\
	& = 
	(-1)^{(r - 1)s} \pounds_\Gamma \mathbf{i}_{\Gamma'} \Omega
	- \mathbf{i}_{\Gamma'} \mathbf{d} \mathbf{i}_\Gamma \Omega
	+ (-1)^r \mathbf{i}_{\Gamma'}\mathbf{i}_\Gamma  \mathbf{d}  \Omega.
\end{align*}
Comparing this with the right-hand side of \eqref{eq1proof}, we see that \eqref{eq1proof} holds if and only if 
\[
	\mathbf{i}_{\Gamma'}\mathbf{i}_\Gamma \mathbf{d} \Omega = 0,
\]	
for all $\Gamma \in \bigwedge^r(TZ)$ and $\Gamma' \in \bigwedge^s(TZ)$.  This is the case if and only if $\mathbf{d}\Omega=0$.
\end{proof}

\paragraph{Contravariant Multi-Dirac Structures.}  Given the fact that a Poisson tensor field gives rise to a Dirac structure (see \cite{Courant1990}), one can now ask whether a multivector field of higher degree induces a multi-Dirac structure.  Unfortunately, this does not hold in full generality: in \cite{Zambon2010} it is shown that if $\Lambda$ is a multivector of degree $k$ on a manifold $Z$, then $\Lambda$ gives rise to a higher-order Dirac structure if and only if either $\Lambda = 0$ or $k$ is equal to 2 (in which case $\Lambda$ is a Poisson bivector) or to the dimension of $Z$.  By the results of section~\ref{sec:hodirac}, the same is true for multi-Dirac structures.

\section{The Multi-Poisson Bracket} \label{sec:multipoisson}

In this section, we analyze some of the properties of the multi-Courant bracket.  We start from an integrable multi-Dirac structure $D$ of degree $n$, and show that the multi-Courant bracket defines a graded Poisson bracket on certain spaces of forms.  This bracket is graded anti-commutative but only satisfies the graded Jacobi identity up to exact forms.

\begin{definition}
Let $D = D_1 \oplus \cdots \oplus  D_n$ be an integrable multi-Dirac structure of degree $n$ on $Z$.
	A $(n - r)$-form $\Sigma$ is said to be \emph{\textbf{admissible}} if there exists an $r$-multivector field $\Gamma_\Sigma$ such that $(\Gamma_\Sigma, \mathbf{d}\Sigma) \in D_r$, where $r = 1, \ldots, n$.
\end{definition}

We define a new grading on the space of admissible forms as follows: let $\Sigma$ be an admissible $s$-form.  We then define the new degree of $\Sigma$ to be given by 
\begin{equation} \label{grading}
	\left| \Sigma \right| := n - s -1.
\end{equation}
The set of all admissible forms of degree $k$ (in the sense of definition \eqref{grading}) is denoted by $\Omega^k_{\text{adm}}$, where $k = 0, \ldots, n - 1$.  

Following \cite{CaIbLe1996}, we define a graded Poisson bracket on the space on admissible forms as follows.  Let $\Sigma \in \Omega^k_{\text{adm}}$ and $\Sigma' \in \Omega^l_{\text{adm}}$ and denote the corresponding multivector fields by $\Gamma_\Sigma$ and $\Gamma_{\Sigma'}$.  We then let
\[
	\{ \Sigma, \Sigma' \} := -(-1)^k\mathbf{i}_{\Gamma_{\Sigma'}}\mathbf{d}\Sigma.
\]

Note that the right-hand side only depends on $\Sigma'$ and not on the particular choice of multivector field $\Gamma_{\Sigma'}$.  Indeed, let $\bar{\Gamma}_{\Sigma'}$ be any other multivector field so that $(\bar{\Gamma}_{\Sigma'}, \mathbf{d}\Sigma') \in D_l$.  Since $(\bar{\Gamma}_{\Sigma'} - \Gamma_{\Sigma'}, 0) \in D_l$, we have that $\left<\!\left< (\bar{\Gamma}_{\Sigma'} - \Gamma_{\Sigma'}, 0), (\Gamma_{\Sigma}, \mathbf{d}\Sigma \right>\!\right>_- = 0$, or 
\[
	\mathbf{i}_{\bar{\Gamma}_{\Sigma'}} \mathbf{d}\Sigma
	=
	\mathbf{i}_{\Gamma_{\Sigma'}} \mathbf{d}\Sigma.
\]

Secondly, let $\Sigma \in \Omega^k_{\text{adm}}$ and $\Sigma' \in \Omega^l_{\text{adm}}$ and note that
\[
	\left| \{ \Sigma, \Sigma' \} \right| = \left| \Sigma \right| 
		+ \left| \Sigma' \right|. 
\]
We now show that the bracket is again admissible.  Indeed, this follows easily from the fact that  
\[
	\bracleft (\Gamma_{\Sigma}, \mathbf{d} \Sigma),  (\Gamma_{\Sigma'}, \mathbf{d} \Sigma') \bracright
		=
	( [\Gamma_\Sigma, \Gamma_{\Sigma'}], (-1)^{kl} \mathbf{d} \{ \Sigma, \Sigma' \} ),
\]
where $k = \left| \Sigma \right|$ and $l = \left| \Sigma' \right|$.  

We show the multi-Poisson bracket is graded anti-commutative.  Let $\Sigma, \Sigma'$ be admissible forms with $\left| \Sigma \right| = k$ and $\left| \Sigma' \right| = l$, and consider the associated multivector fields $\Gamma_\Sigma$ and $\Gamma_{\Sigma'}$.  We then have that 
\begin{align*}
 \left<\!\left< (\Gamma_\Sigma, \mathbf{d} \Sigma), 
	(\Gamma_{\Sigma'}, \mathbf{d} \Sigma') \right>\!\right>_- 
	& = \frac{1}{2} \left( \mathbf{i}_{\Gamma'} \mathbf{d} \Sigma - (-1)^{(k-1)(l-1)} \mathbf{i}_{\Gamma}\mathbf{d} \Sigma' \right) \\
	& = \frac{(-1)^k}{2} \left( \{\Sigma, \Sigma'\} + (-1)^{kl} 
	 	\{\Sigma', \Sigma\} \right),
\end{align*}
but since the left-hand side is zero, we conclude that 
\begin{equation}
\{\Sigma, \Sigma'\} = -(-1)^{kl} 
	 	\{\Sigma', \Sigma\}.
\end{equation}

A natural question now is whether the multi-Poisson bracket satisfies the graded Jacobi identity.  This turns out to not be the case, even when the multi-Dirac structure is integrable, as is shown in the following lemma.

\begin{lemma}
	Consider an integrable multi-Dirac structure $D$ of degree $n$.
	Let $\Sigma \in \Omega^k_{\text{adm}}$,  $\Sigma' \in \Omega^l_{\text{adm}}$, and $\Sigma'' \in \Omega^m_{\text{adm}}$ be admissible forms.  The multi-Poisson bracket satisfies the relation
	\begin{equation} \label{gradedjacobi}
	\begin{split}
		 	(-1)^{km} \{ \{ \Sigma, \Sigma'\}, \Sigma'' \}
			+  
			(-1)^{lk} \{ \{ \Sigma', \Sigma''\}, \Sigma\}
						+  
			(-1)^{ml} \{ \{ \Sigma'', \Sigma\}, \Sigma'\} \\
			=
			(-1)^{(k + l)(m - 1)} \mathbf{d} 
			\mathbf{i}_{\Gamma_{\Sigma'}}
			\mathbf{i}_{\Gamma_{\Sigma''}} \mathbf{d}\Sigma,  
	\end{split}
	\end{equation}
	where $\Gamma_{\Sigma}$, $\Gamma_{\Sigma'}$ and $\Gamma_{\Sigma''}$ are the multivector fields associated to $\Sigma$, $\Sigma'$ and $\Sigma''$, respectively.
\end{lemma}

\begin{proof}
This follows easily from the observation that 
\[
	\{ \{ \Sigma, \Sigma' \}, \Sigma'' \}  =
	(-1)^{k + l} \mathbf{i}_{\Gamma_{\Sigma''}} \mathbf{d} \{\Sigma, \Sigma' \}  = (-1)^k \mathbf{i}_{\Gamma_{\Sigma''}} \mathbf{d}
		\mathbf{i}_{\Gamma_{\Sigma'}} \mathbf{d} \Sigma
	 = (-1)^k \mathbf{i}_{\Gamma_{\Sigma''}} 
		\pounds_{\Gamma_{\Sigma'}} \mathbf{d} \Sigma
\]
and the expression \eqref{integrability} for $T_D$: putting everything together, we have that 
\begin{multline*}
	T_D( (\Gamma_{\Sigma}, \mathbf{d} \Sigma), 
	(\Gamma_{\Sigma'}, \mathbf{d} \Sigma'), 
	(\Gamma_{\Sigma''}, \mathbf{d} \Sigma'') ) = \\
	-(-1)^{k(m-1)} \Big( 
	(-1)^{km} \{ \{ \Sigma, \Sigma'\}, \Sigma'' \}
			+  
			(-1)^{lk} \{ \{ \Sigma', \Sigma''\}, \Sigma\}
						+  
			(-1)^{ml} \{ \{ \Sigma'', \Sigma\}, \Sigma'\} \\
		- (-1)^{(k + l)(m - 1)} \mathbf{d} 
			\mathbf{i}_{\Gamma_{\Sigma'}}
			\mathbf{i}_{\Gamma_{\Sigma''}} \mathbf{d}\Sigma \Big),
\end{multline*}
while the left-hand side vanishes as $D$ is integrable.
\end{proof}

We summarize the properties of the multi-Poisson bracket in the following theorem: 
\begin{theorem}
	The multi-Poisson bracket is graded anti-commutative, and satisfies the graded Jacobi identity up to exact forms. 
\end{theorem} 

The fact that the graded Jacobi identity is satisfied only up to exact forms is not new and arises also when considering multi-Poisson brackets associated to multisymplectic structures; see \cite{FoPaRo2005}.  Following a similar procedure as in \cite{CaIbLe1996}, we may now consider the quotient space of admissible forms modulo exact forms (which are trivially admissible).  The multi-Poisson bracket drops to this space, and endows it with the structure of a graded Lie algebra. 

Another interpretation was given by \cite{BaHoRo2010} (see also the work of \cite{Rogers2010b} and \cite{Zambon2010}), where the exact forms on the right-hand side of \eqref{gradedjacobi} signal the fact that the space of admissible forms has the structure of an {\bfi $L_\infty$-algebra}.


\section{Outlook and Future Work}

In this paper, we have focused on the definition of multi-Dirac structures and the induced Gerstenhaber algebra on the space of sections.  In a previous paper, we have shown that multi-Dirac structures play a similar role in classical field theory as standard Dirac structures in mechanics.  Much remains to be done, and we now give an outline of a few interesting questions related to the geometry of multi-Dirac structures as well as their applications in classical field theory.

\begin{itemize}
	\item Many geometric structures have a straightforward, natural interpretation in terms of \emph{graded geometry} (see \cite{CaZa2009}).  Given that multi-Dirac structures appear as a graded analogue of standard Dirac structures, a natural question is therefore whether multi-Dirac structures can be understood in terms of graded geometry as well.
	
	\item Multi-Dirac structures are in a one-to-one correspondence to higher-order Dirac structures (see \cite{Zambon2010}), but their algebraic properties seem at first sight to be different.  While the space of sections of a multi-Dirac structure is equipped with the structure of a Gerstenhaber algebra, higher-order Dirac structures instead give rise to \emph{$L_\infty$-algebras}.  It would be of considerable interest to understand the link between these two algebraic structures directly.  
	
	\item In the context of classical field theories, a few questions naturally arise.  First of all, in addition to the multi-Poisson brackets in section~\ref{sec:multipoisson}, there are several other kinds of brackets that can be used to write down the field equations (see \cite{MaMoMo1986, MarcoMarsden, FoRo2005} and the references therein).  Furthermore, in \cite{Bridges05} multi-symplectic structures were considered on the exterior algebra of a Riemannian manifold. It would be interesting to see if these brackets and structures induce corresponding multi-Dirac structures.  
	
	In this context we also mention the concept of \emph{Stokes-Dirac structures} (see \cite{VaMa2002}), where the calculus of forms on a manifold and the Stokes theorem are used to develop a concept of Dirac structures suitable for field theories.  In \cite{VaYoLeMa2010}, we have shown that these structures arise through Poisson reduction of infinite-dimensional structures.  When a space-time splitting is chosen, it can be shown that a multi-Dirac structure induces an infinite-dimensional Dirac structure on the space of fields.  The precise relation between both points of view and the implications for the different kinds of Poisson brackets mentioned before will be the subject of forthcoming work.
	
	One further question relates to the formulation of the field equations of classical field theory.  When a multi-Dirac structure $D=D_1 \oplus \cdots \oplus D_n$ on the  jet extension $J^1 Y$ of a fiber bundle $\pi : Y \to X$ is given, we showed in \cite{VaYoLeMa2011} that the field equations can then be expressed in terms of $D_n$.  It is not clear, however, what role is played by the other components $D_r$, with $r < n$.
	
One answer might be found in the link with admissible forms and multi-Poisson brackets of section~\ref{sec:multipoisson}.  Admissible forms of degree $n$ can be integrated over spatial hypersurfaces in $X$ and correspond to classical observables of the theory (see \cite{kijowski}).  In the same vein, admissible forms of lower degree might correspond to observables with support on lower-dimensional submanifolds.

\end{itemize}

\appendix
\section{Properties of Multivector Fields and Generalized Lie Derivatives} \label{sec:lie}

In this appendix we give a quick account of multivector fields and their properties.  We follow the exposition and the sign conventions for multivectors and Lie derivatives of \cite{FoPaRo2005} and \cite{Marle1997}.

\paragraph{Multivector Fields.}

Let $M$ be a manifold.  A {\bfi multivector field} of degree $k$ on $M$ is a section of the $k$-th exterior power $\bigwedge^k (TM)$.  We say that a $k$-multivector field $\Gamma$ is {\bfi decomposable} if it can be written as a wedge product of $k$ vector fields $X_1, \ldots, X_k$ on $M$:
\[
	\Gamma = X_1 \wedge \cdots \wedge X_k.
\]
We denote the space of all $k$-multivector fields on $M$ by $\mathfrak{X}^k(M)$.  The {\bfi Schouten-Nijenhuis bracket} is the unique $\mathbb{R}$-bilinear map 
\[
	[\cdot, \cdot] : \mathfrak{X}^k(M) \times \mathfrak{X}^l(M)
		\rightarrow
		\mathfrak{X}^{k + l - 1}(M)
\]
with the following properties: for multivector fields $\Gamma, \Gamma'$ and $\Gamma''$ of degree $k, l$ and $m$ respectively, we have that
\begin{enumerate}
\item The Schouten-Nijenhuis bracket vanishes on functions: for $f, g \in C^\infty(M)$, we have $[f, g] = 0$; 

\item It is graded anticommutative: 
\[
	[\Gamma, \Gamma'] = - (-1)^{(k - 1)(l - 1)} [\Gamma', \Gamma];
\]
\item It coincides with the Lie bracket when restricted to vector fields; 
\item It satisfies the graded Leibniz rule: 
\[
	[\Gamma, \Gamma' \wedge \Gamma''] 
	= 
	[\Gamma, \Gamma'] \wedge \Gamma'' 
	+ (-1)^{(k-1)l} \Gamma' \wedge [\Gamma, \Gamma'']; 
\]
\item It satisfies the graded Jacobi identity:
\[
	(-1)^{(k-1)(m-1)} [\Gamma, [\Gamma', \Gamma'']] + 
		\text{cyclic perm.} = 0.
\]
\end{enumerate}

It follows that the set $\mathfrak{X}(M)$ of all multivector fields equipped with the Schouten-Nijenhuis bracket forms a {\bfi Gerstenhaber algebra} (see \cite{Ge1963}).

\paragraph{The Lie Derivative.}

  We define the {\bfi contraction} of a decomposable multi-vector field $\Gamma = X_1 \wedge \cdots \wedge X_k$ and a differential form $\alpha$ to be
\[
	\mathbf{i}_\Gamma \alpha = \mathbf{i}_{X_k} \ldots \mathbf{i}_{X_1} \alpha.
\]
This operation can then be extended to arbitrary, non-decomposable multi-vector fields by linearity.  The {\bfi Lie derivative} $\pounds_\Gamma \alpha$ of the form $\alpha$ along a multi-vector field $\Gamma$ of degree $k$ is defined by means of a generalization of Cartan's formula:
\[
	\pounds_\Gamma \alpha = \mathbf{d} \mathbf{i}_\Gamma
		- (-1)^k \mathbf{i}_\Gamma \mathbf{d} \alpha, 
\]
where $k$ is the degree of $\Gamma$.  The Lie derivative satisfies the following useful properties:

\begin{proposition}[Prop.~A.3 in \cite{FoPaRo2005}]
	For any two multi-vector fields $\Gamma$ and $\Gamma'$ of degree $k$ and $l$ respectively, and any differential form $\alpha$, we have 
	\begin{gather}
		\mathbf{d} \pounds_\Gamma \alpha = (-1)^{k - 1} \pounds_\Gamma \mathbf{d} \alpha \\[2ex]
		\label{koszulid}
		\mathbf{i}_{[\Gamma, \Gamma']} \alpha = (-1)^{(k - 1)l} 
			\pounds_\Gamma \mathbf{i}_{\Gamma'} \alpha -
			\mathbf{i}_{\Gamma'} \pounds_{\Gamma} \alpha \\[2ex]
		\pounds_{[\Gamma, \Gamma']} \alpha = 
			(-1)^{(k - 1)(l - 1)} \pounds_\Gamma \pounds_{\Gamma'} \alpha
			- \pounds_{\Gamma'} \pounds_{\Gamma} \alpha \\[2ex]
		\pounds_{\Gamma \wedge \Gamma'} \alpha = 
			(-1)^l \mathbf{i}_{\Gamma'} \pounds_\Gamma \alpha 
			+ \pounds_{\Gamma'} \mathbf{i}_\Gamma \alpha.
	\end{gather}
\end{proposition}
Throughout the paper, \eqref{koszulid} is referred to as the {\bfi Koszul identity}.

\section{The Integrability Tensor $T_D$} \label{app:tensor}

In this appendix, we show that the expression \eqref{inttensor} for $T_D$ determines a tensor, by extending the argument due to \cite{Courant1990} to the case of multi-Dirac structures.  Straightforward algebraic manipulations show that $T_D$ is the restriction to $D$ of the following form:
\begin{gather*}
	T((\Gamma, \Sigma), (\Gamma', \Sigma'), (\Gamma'', \Sigma''))  = (-1)^{(r-1)t} \Big[
	-\frac{1}{3}  (-1)^{rs} \mathbf{d} \mathbf{i}_{\Gamma}  \left<\!\left< (\Gamma', \Sigma'), (\Gamma'', \Sigma'') \right>\!\right>_+ \\
	+ (-1)^{r(s-1)} \left(  \mathbf{i}_\Gamma  \mathbf{d}  \left<\!\left< (\Gamma', \Sigma'), (\Gamma'', \Sigma'') \right>\!\right>_+ - (-1)^t \mathbf{i}_{\Gamma} \mathbf{i}_{\Gamma''} \mathbf{d} \Sigma' \right) 
		+ \text{(c. p.)} \Big],
\end{gather*}
where $\text{(c. p.)}$ stands for cyclic permutations of the terms between square brackets.  It is not hard to show that $T$ is graded anticommutative in its arguments:
\[
	T((\Gamma, \Sigma), (\Gamma'', \Sigma''), (\Gamma, \Sigma))	=
	- (-1)^{(s - 1)(t - 1)} T((\Gamma, \Sigma), (\Gamma', \Sigma'), (\Gamma'', \Sigma'')), 
\]
and similarly for permutations of the other arguments.  From the definition \eqref{inttensor}, we have now that $T_D$ is obviously $\mathcal{F}(Z)$-linear in its first argument, and as $T_D$ is the restriction to $D$ of a graded anticommutative expression, it must be linear in the other arguments as well.  We conclude that $T_D$ is a tensor.


\end{document}